\documentclass[a4paper,oneside, reqno]{amsart}
\usepackage{mathrsfs}
\usepackage{amsfonts}
\usepackage{amssymb}
\usepackage{amsxtra}
\usepackage{dsfont}
\usepackage{amsthm}

\usepackage[english]{babel}
\usepackage{enumitem}
\usepackage{hyperref}
\usepackage{stmaryrd}

\usepackage{enumitem}

\makeatletter
\@namedef{subjclassname@2020}{%
  \textup{2020} Mathematics Subject Classification}
\makeatother

\usepackage[T1]{fontenc}

\theoremstyle{plain}
\newtheorem{thm}{Theorem}[section]
\newtheorem{cor}[thm]{Corollary}

\newtheorem{prop}[thm]{Proposition}

\newtheorem{conj}[thm]{Conjecture}
\theoremstyle{remark}
\newtheorem{rem}{Remark}

\numberwithin{equation}{section}

\newcommand{\CC}{\mathbb{C}}
\newcommand{\NN}{\mathbb{N}}

\newcommand{\tf}{\tilde{f}}
\newcommand{\tg}{\tilde{g}}
\renewcommand{\th}{\tilde{h}}

\newcommand{\ind}[1]{{\mathds{1}_{{#1}}}}

\def\R{{\mathbb{R}}}

\def\C{{\mathbb{C}}}
\def\N{{\mathbb{N}}}

\DeclareMathOperator{\sgn}{sgn}
\DeclareMathOperator{\Deg}{Deg}
\DeclareMathOperator{\Tail}{Tail}

\DeclareMathOperator{\Rea}{Re}

\title[Dimension-free estimates on the Hamming cube]{Dimension-free estimates for low degree functions  on the Hamming cube}


\author[Domelevo]{Komla Domelevo}
\address[Komla Domelevo]{
Faculty of Mathematics and Computer Science\\
Institute of Mathematics\\
Emil-Fischer-Str. 41\\
97074 Würzburg\\
Germany
}
\email{komla.domelevo@mathematik.uni-wuerzburg.de}

\author[Durcik]{Polona Durcik}
\address[Polona Durcik]{Schmid College of Science and Technology\\
Chapman University\\
One University Drive\\
Orange, CA 92866, USA}
\email{durcik@chapman.edu}

\author[Fragkiadaki]{Valentia Fragkiadaki}
\address[Valentia Fragkiadaki]{Department of Mathematics\\ Texas A\&M University\\
College station\\
TX 77843, USA}
\email{valeria96@tamu.edu}

\author[Klein]{Ohad Klein}
\address[Ohad Klein]{School of Computer Science and Engineering\\
Hebrew University of Jerusalem\\
Jerusalem\\
Israel}
\email{ohadkel@gmail.com}

\author[Oliveira e Silva]{Diogo Oliveira e Silva}
\address[Diogo Oliveira e Silva]{ 
Center for Mathematical Analysis, Geometry and Dynamical Systems \&
Departamento de Matem\'{a}tica\\ 
Instituto Superior T\'{e}cnico\\
Av. Rovisco Pais\\ 
1049-001 Lisboa, Portugal} 
\email{diogo.oliveira.e.silva@tecnico.ulisboa.pt}

\author[Slav\'ikov\'a]{Lenka Slav\'ikov\'a}
\address[Lenka Slav\'ikov\'a]{Department of Mathematical Analysis\\
Faculty of Mathematics and Physics\\
Charles University\\
Sokolovsk\'a 83, 186 75 Praha 8, Czech Republic}
\email{slavikova@karlin.mff.cuni.cz}

\author[Wr{\'o}bel]{B{\l}a{\.z}ej Wr{\'o}bel}
\address[B{\l}a{\.z}ej Wr{\'o}bel]{
Institute of Mathematics\\
	of the Polish Academy of Sciences\\
	Śniadeckich 8\\
	00-656 Warsaw\\
	Poland \& Institute of Mathematics\\
	University of Wrocław\\
	Plac Grun\-waldzki 2\\
	50-384 Wrocław\\
	Poland}
\email{blazej.wrobel@math.uni.wroc.pl}

\begin{document}

\begin{abstract}
The main result of this paper  are dimension-free $L^p$ inequalities, $1<p<\infty$, for low degree scalar-valued  functions on the Hamming cube. More precisely, for any $p>2,$ $\varepsilon>0,$ and $\theta=\theta(\varepsilon,p)\in (0,1)$ satisfying
   \[
   \frac{1}{p}=\frac{\theta}{p+\varepsilon}+\frac{1-\theta}{2}
   \]
   we obtain, for any function $f:\{-1,1\}^n\to \C$ whose spectrum is bounded from above by $d,$  the Bernstein-Markov type inequalities
 $$
\|\Delta^k f\|_{p} \le C(p,\varepsilon)^k \,d^k\, \|f\|_{2}^{1-\theta}\|f\|_{p+\varepsilon}^{\theta},\qquad k\in \NN.$$
Analogous inequalities are also proved for $p\in (1,2)$ with $p-\varepsilon$ replacing $p+\varepsilon.$ 
As a corollary, if $f$ is Boolean-valued or $f\colon \{-1,1\}^n\to \{-1,0,1\},$ we obtain the bounds
$$\|\Delta^k f\|_{p} \le C(p)^k \,d^k\, \|f\|_p,\qquad k\in \NN.$$ At the endpoint $p=\infty$  we provide counterexamples for which a linear growth in $d$ does not suffice when $k=1$.

We also obtain a counterpart of this result on tail spaces. Namely, for $p>2$ we prove that any function $f:\{-1,1\}^n\to \C$ whose spectrum is bounded from below by $d$ satisfies  the  upper bound on the decay of the heat semigroup
 $$ \|e^{-t\Delta}f\|_{p} \le \exp(-c(p,\varepsilon) td) \|f\|_{2}^{1-\theta}\|f\|_{p+\varepsilon}^{\theta},\qquad t>0,$$
and an analogous estimate for $p\in (1,2).$

The constants $c(p,\varepsilon)$ and $C(p,\varepsilon)$ depend only on $p$ and $\varepsilon$;
crucially, they are independent of the dimension $n$. 
    \end{abstract}

    \maketitle

\section{Introduction}

Consider the Hamming cube $\{-1,1\}^n$ equipped with the uniform probability measure. Given $p\in [1,\infty)$, the $L^p$ norm of a function $f\colon \{-1,1\}^n\to \CC$ is given by
\[
\|f\|_{L^p(\{-1,1\}^n)}:=\left(\frac{1}{2^n}\sum_{\delta\in \{-1,1\}^n}|f(\delta)|^p\right)^{1/p}.
\]
We also set $\|f\|_{L^{\infty}(\{-1,1\}^n)}:=\max_{\delta\in \{-1,1\}^n} |f(\delta)|.$ The symbol $\langle f,g\rangle$ denotes the (Hermitian, complex) inner product on $L^2(\{-1,1\}^n).$

Any $f\colon \{-1,1\}^n\to \CC$ can be represented as its Hamming cube Fourier--Walsh expansion via
\begin{equation}
	\label{eq:fhf}
	f=\sum_{S\subseteq [n]}\widehat{f}(S) w_S, 
\end{equation} 
with $w_S$ being the Walsh functions $w_S(\delta)=\prod_{j\in S}\delta_j,$ $\delta=(\delta_1,\ldots,\delta_n)\in \{-1,1\}^n$, $  S\subseteq [n]:=\{1,\ldots,n\}.$
The coefficients $\widehat{f}(S)$ in \eqref{eq:fhf} are the Fourier coefficients of $f$ given by
\[
\widehat{f}(S):=\frac{1}{2^n}\sum_{\delta\in \{-1,1\}^n} f(\delta) w_S(\delta).
\]
Given $d\in [n]$,    let
\begin{align*}
  \Deg_{\le d}&=\{f\colon\{-1,1\}^n\to \mathbb C \colon f=\sum_{|S|\le d }\widehat{f}(S) \,w_S\}
\end{align*}
denote the space of functions of degree at most $d$. In other words, $\Deg_{\le d}$ contains those functions  $f\colon\{-1,1\}^n\to \CC$ for which $\widehat{f}(S)=0$ if $|S|>d$.

Define the Hamming cube Laplacian by
\begin{equation}
	\label{eq:hala}
	\Delta f=\sum_{j=1}^n \partial_j f,
 \end{equation}
 with
 \begin{equation}
 \label{eq:pala} \partial_jf(\delta)=\frac{f(\delta_1,\ldots,\delta_j,\ldots,\delta_n)-f(\delta_1,\ldots,-\delta_j,\ldots,\delta_n)}{2}.
\end{equation}
In this way, $w_S$ is an eigenfunction of $\Delta$ with eigenvalue $|S|.$ Thus,   if $f$ is decomposed as in \eqref{eq:fhf}, then $$\Delta f=\sum_{S\subseteq [n]} |S|\, \widehat{f}(S)\,w_S.$$ This leads to the following definition of the heat semigroup for $f$ given by \eqref{eq:fhf}:
\begin{equation*}
	e^{-t\Delta}f=\sum_{S\subseteq [n]}e^{-t|S|}\,\widehat{f}(S) \,w_S.
\end{equation*}
When {$p=2$}, it is straightforward to verify via Parseval's identity that
\begin{equation}
\label{eq:BM l2}
\|\Delta f\|_{L^2(\{-1,1\}^n)} \le d\, \|f\|_{L^2(\{-1,1\}^n)}
\end{equation}
whenever $f\in \Deg_{\le d}.$

The main purpose of this note is to establish a {variant} of \eqref{eq:BM l2} on $L^p$ spaces when {$p\neq 2$. As a consequence, we will give partial answers to the following conjectures formulated by Eskenazis--Ivanisvili in \cite[Sections 1.2.2--1.2.3]{EI1} in the case when $f$ takes values in $\{-1,0,1\}$ and $X=\mathbb{C}$. The definition of a $K$-convex Banach space $X$ and the $L^p(\{-1,1\}^n;X)$-norm can be recalled from  \cite[Section 1.2]{EI1}.}

\begin{conj}
\label{conj 1}
Let $(X,\|\cdot\|_X)$ be a $K$-convex Banach space. For every $p\in (1,\infty),$ there exists $C(p,X)$ such that, for every $d\in [n]$ and every $f\in \Deg_{\le d}$,  
\begin{equation}
 \label{eq: conj1}   
 \|\Delta f\|_{L^p(\{-1,1\}^n;X)} \le C(p,X)\cdot d\, \|f\|_{L^p(\{-1,1\}^n;X)}.
\end{equation}
\end{conj}

\begin{conj}
\label{conj 2}
Let $(X,\|\cdot\|_X)$ be a $K$-convex Banach space. For every $p\in (1,\infty),$ there exist constants $c(p,X)$ and $C(p,X)$ such that, for every $d\in [n]$, every $f\in \Deg_{\le d}$ and every $t\geq 0$, 
\begin{equation}
 \label{eq: conj2}   
 \|e^{-t\Delta} f\|_{L^p(\{-1,1\}^n;X)} \ge c(p,X)\cdot \exp(-C(p,X) \cdot td) \|f\|_{L^p(\{-1,1\}^n;X)}.
\end{equation}
\end{conj}
\noindent In \cite{EI1} the authors observed that Conjecture \ref{conj 1} implies Conjecture \ref{conj 2} with $c(p,X)=1$; see Remark \ref{rem: imp} for details. It turns out that also Conjecture \ref{conj 2} implies Conjecture \ref{conj 1}, see Remark \ref{rem: Con21}, so that these conjectures are in fact equivalent. This observation emerged from a discussion with Alexandros Eskenazis.

In what follows, we set
\begin{equation}
\label{eq: peps}
p_{\varepsilon}=\begin{cases}p+\varepsilon,\qquad p\ge 2,\quad\varepsilon>0
\\p-\varepsilon,\qquad p\in(1,2),\quad \varepsilon\in (0,p-1)
\end{cases};
\end{equation}
clearly, $p_{\varepsilon}\to p$ as $\varepsilon\to 0^+.$
The following is the main theorem of our note.

\begin{thm}
\label{t: main p imp}
Let $p\in (1,\infty)$ and let $\theta=\theta(\varepsilon,p)\in (0,1)$ satisfy
 \begin{equation}
 \label{eq: tpe}
   \frac{1}{p}=\frac{\theta}{p_{\varepsilon}}+\frac{1-\theta}{2}.
   \end{equation}
  Take  $d\in [n]$  and let $f\in \Deg_{\le d}$. Then, for any $k\in \N$ we have
\begin{equation}
 \label{eq: main p imp}   
 \|\Delta^k f\|_{L^p(\{-1,1\}^n)} \le C(p,\varepsilon)^k  \cdot d^k\, \|f\|_{L^2(\{-1,1\}^n)}^{1-\theta}\|f\|_{L^{p_{\varepsilon}}(\{-1,1\}^n)}^{\theta},
\end{equation}
for a constant $C(p,\varepsilon)$ depending on $p, \varepsilon$ but not on $d,k$ nor on the dimension $n.$ Consequently, for any $p\ge 2$ and $\varepsilon>0$ we also have
\begin{equation}
 \label{eq: main p imp'}   
 \|\Delta^k f\|_{L^p(\{-1,1\}^n)} \le C(p,\varepsilon)^k  \cdot d^k\,\|f\|_{L^{p+\varepsilon}(\{-1,1\}^n)}.
\end{equation}
\end{thm}

The proof of Theorem \ref{t: main p imp} is based on two ingredients: Hadamard's three-lines theorem and dimension-free $L^p$ bounds for the imaginary powers $\Delta^{iu}$. We are not aware of any previous application of these results to the analysis of functions on the Hamming cube. Our proof is similar to the proofs of both Riesz--Thorin interpolation theorem and Stein's interpolation theorem for analytic families of operators (see e.g. \cite[Sections 1.3.2 and 1.3.3]{gra} and \cite{St1}), with a twist. Indeed, we apply Hadamard's theorem in a slightly different way, by removing an instance of one single letter of the alphabet: instead of the holomorphic function $z\mapsto \langle \Delta^z f^z,g^z \rangle$, we consider a similar function in which $f^z$ does not depend on $z,$ namely $\langle \Delta^z f,g^z \rangle.$ This comes at the cost of obtaining $\|f\|_{L^2(\{-1,1\}^n)}^{1-\theta}\|f\|_{L^{p_{\varepsilon}}(\{-1,1\}^n)}^{\theta}$ instead of $\|f\|_{L^{p}(\{-1,1\}^n)}$ on the right hand side of \eqref{eq: main p imp}. It is vital that $\Delta$ generates a symmetric contraction semigroup, and thus its imaginary powers $\Delta^{iu},$ $u\in \R,$ satisfy $L^p$  bounds, $1<p<\infty$, with explicit constants which depend only on $p$ and $u$; see e.g. \cite{Co, topics}. At the $L^2$ endpoint we merely use spectral properties of $\Delta^z,$ $\Rea z\ge 0,$ on $L^2,$ which imply
\begin{equation}
\label{eq: Delz}
 \|\Delta^z f\|_{L^2(\{-1,1\}^n)} \le  d^{\Rea z}\, \|f\|_{L^2(\{-1,1\}^n)},\,\text{ for all } f\in \Deg_{\le d}.
\end{equation}

Theorem \ref{t: main p imp}  is proved in detail in Section \ref{sec: pmain imp}.
We proceed to discuss its consequences.

The first corollary is a lower bound on the decay of the heat semigroup acting on low degree functions. Corollary \ref{t: helo} below is equivalent to the inequality
\[
 \|e^{t\Delta} f\|_{L^p(\{-1,1\}^n)} \le \exp(C(p,\varepsilon)\cdot  td) \|f\|_{L^{p+\varepsilon}(\{-1,1\}^n)}, 
\]
for all $f\in \Deg_{\le d}$.
This inequality is a consequence of \eqref{eq: main p imp'} from Theorem \ref{t: main p imp} via a Taylor expansion argument.

\begin{cor}
\label{t: helo}
Let $p\in [2,\infty),$ $\varepsilon>0,$ and  $d\in [n].$ Then, for $f\in \Deg_{\le d}$ and all $t>0$,  
\begin{equation}
 \label{eq: helo}   
 \|e^{-t\Delta}f\|_{L^{p+\varepsilon}(\{-1,1\}^n)} \ge \exp(-C(p,\varepsilon)\cdot   td) \|f\|_{L^p(\{-1,1\}^n)},
\end{equation}
where $C_p$ is the constant from \eqref{eq: main p imp'}   .
\end{cor}

\begin{rem}
\label{rem: imp}
It was observed in \cite[Remark 33]{EI1} that Conjecture \ref{conj 1} implies Conjecture \ref{conj 2}. Indeed, since the Hamming cube Laplacian $\Delta$ preserves the space $\Deg_{\leq d}$, a repeated application of \eqref{eq: conj1} yields
\[ \|\Delta^k f\|_{L^p(\{-1,1\}^n;X)} \le C(p,X)^k \cdot d^k\, \|f\|_{L^p(\{-1,1\}^n;X)},\qquad k\in \N.\]
From this inequality, we easily obtain via Taylor expansion that
\[
 \|e^{t\Delta} f\|_{L^p(\{-1,1\}^n;X)} \le \exp(C(p,X) \cdot td) \|f\|_{L^p(\{-1,1\}^n;X)}, \, \text{ for all } f\in \Deg_{\le d},
\]
which is equivalent to Conjecture \ref{conj 2} with $c(p,X)=1.$

However, the above argument cannot be applied for concluding that the case $k=1$ of \eqref{eq: main p imp'}  implies Corollary \ref{t: helo}. Iterating the case $k=1$ of \eqref{eq: main p imp'} as a bound between $L^{p+\varepsilon/k}$ and $L^{p+\varepsilon/(k+1)}$ spaces, we may prove that
\[
\|\Delta^k f\|_{L^p(\{-1,1\}^n)} \le \tilde{C}(p,\varepsilon,k)  \cdot d^k\,\|f\|_{L^{p+\varepsilon}(\{-1,1\}^n)},\qquad k\in\N,
\]
for some constant $\tilde{C}(p,\varepsilon,k)$ depending on $p,k,$ and $\varepsilon.$ Unfortunately, it is not obvious whether $\tilde{C}(p,\varepsilon,k) $ has the desired exponential bound in $k$ given by $C(p,\varepsilon)^k$. In summary, in order to deduce Corollary \ref{t: helo} we do need to prove \eqref{eq: main p imp'} for general $k$.
\end{rem}
\begin{rem}
\label{rem: Con21}
One can also prove that Conjecture \ref{conj 2} implies Conjecture \ref{conj 1}. A deep result of Pisier \cite{Pi1} implies that for each $p\in(1,\infty)$ and any $K$-convex Banach space $(X,\|\cdot\|_X)$ the semigroup $\{e^{-t\Delta}\}_{t\ge 0}$ is analytic on $L^p(\{-1,1\}^n;X).$  It is known, see e.g. \cite[Theorem 4.6 c)]{EnNa}, that analyticity implies the estimate
\begin{equation}
\label{eq: holo ineq}
\|t\Delta e^{-t\Delta} f\|_{L^p(\{-1,1\}^n;X)} \le \tilde{C}(p,X)\|f\|_{L^p(\{-1,1\}^n;X)},
\end{equation}
for all $t>0$ and $f\in L^p(\{-1,1\}^n;X)$ with $\tilde{C}(p,X)$ depending only on $p$ and $X.$ Take now $f\in \Deg_{\le d}$ so that $\Delta f\in  \Deg_{\le d}.$ Thus Conjecture \ref{conj 2} (with $t=1/d$ and $\Delta f$ in place of $f$) and \eqref{eq: holo ineq} (with $t=1/d$) imply
\begin{equation*}
\begin{split}
\|\Delta  f\|_{L^p(\{-1,1\}^n;X)} &\le c(p,X)\exp(C(p,X))\|\Delta e^{-\frac1d\Delta} f\|_{L^p(\{-1,1\}^n;X)}\\
&\le c(p,X)\exp(C(p,X)) \tilde{C}(p,X)\cdot d\cdot  \| f\|_{L^p(\{-1,1\}^n;X)},
\end{split}
\end{equation*}
which is \eqref{eq: conj1} with the constant $c(p,X)\exp(C(p,X)) \tilde{C}(p,X).$

\end{rem}

As a further consequence of Theorem \ref{t: main p imp} we obtain a dimension-free Bernstein-Markov type inequality for all functions $f$ that take values in $\{-1,0,1\}.$ Corollary \ref{t: corma} below follows from \eqref{eq: main p imp} (with $k=1$ and $\varepsilon=(p-1)/2$) together with the observation that, for such functions $f$, 
\[\|f\|_{L^{p_{\varepsilon}}(\{-1,1\}^n)}^{p_{\varepsilon}}=\|f\|_{L^2(\{-1,1\}^n)}^2=\|f\|_{L^p(\{-1,1\}^n)}^p.\]

\begin{cor}
\label{t: corma}
Let $p\in (1,\infty)$ and  $d\in [n].$ Take $f\in \Deg_{\le d}$ such that $f\colon \{-1,1\}^n\to\{-1,0,1\}.$ Then we have
\begin{equation}
 \label{eq: corma}   
 \|\Delta f\|_{L^p(\{-1,1\}^n)} \le C_p \cdot d\, \|f\|_{L^p(\{-1,1\}^n)}
\end{equation}
for some constant $C_p$ depending on $p$ but not on the dimension $n.$
\end{cor}

It is worth noting that Theorem \ref{t: main p imp} cannot be deduced from interpolation between the trivial $L^2$ bound and an $L^{\infty}$ bound. Indeed, case $k=1$ of \eqref{eq: main p imp'}  does not hold at the endpoint $p=p+\varepsilon=\infty.$ The following proposition is contained in \cite{EI1}, see \cite[proof of Theorem 5, p.\ 254]{EI1} and \cite[Remark 24]{EI1}, but we repeat the proof here for the convenience of the reader. 
\begin{prop}[\cite{EI1}]
    \label{lem: count 1}
Given $n\in \N$ and $d\in [n]$, let ${\bf C}(n,d)$ denote the best constant in the inequality
\begin{equation*} 
 \|\Delta f\|_{L^{\infty}(\{-1,1\}^n)} \le {\bf C}(n,d)\cdot \|f\|_{L^{\infty}(\{-1,1\}^n)},\qquad f\in \Deg_{\le d}.
\end{equation*}
Then it holds that $\liminf_{n\to \infty} {\bf C}(n,d) \ge {d^2.}$
\end{prop}
\noindent Proposition \ref{lem: count 1} is justified in Section  \ref{sec: pcon}. Its proof is based on  properties of Chebyshev polynomials.

Corollary \ref{t: corma}  also does not hold at the endpoint $p=\infty$. A counterexample is given by the Kushilevitz function; see \cite[Section 6.3]{HKP} and \cite{AW}. This was brought to our attention by Paata Ivanisvili.
\begin{prop}
    \label{lem: count 2}
Let $k\in \NN$ and take  $n=6^k$ and $d=3^k.$ Then there exists a function $f\colon \{-1,1\}^n\to \{-1,1\}$ such that $f\in \Deg_{\le d}$ but
\begin{equation*}  
 \|\Delta f\|_{L^{\infty}(\{-1,1\}^n)} \ge  d^{\log 6/ \log 3 } \|f\|_{L^{\infty}(\{-1,1\}^n)}.
\end{equation*}
\end{prop}
\noindent  Proposition  \ref{lem: count 2} is proved in Section \ref{sec: pcon}. Note that $\log 6/ \log 3\approx 1.63$, and it would be interesting to see if one can improve the growth to a constant times $ d^2$ as in Proposition \ref{lem: count 1}. {Since we always have the trivial upper bound $\|\Delta f\|_{L^{\infty}(\{-1,1\}^n)} \le  n \|f\|_{L^{\infty}(\{-1,1\}^n)}$ this would require finding a function of degree $d= O(\sqrt{n}),$ whereas the degree of the Kushilevitz function is $d=n^{\log 3/\log6}\ge n^{0.61}.$}  
Note that unlike the $p=\infty$ case, Corollary~\ref{t: corma} does hold when $p=1.$ 
Namely, since $\partial_j f$ takes values in $\{-1,-1/2,0,1/2,1\}$ for $f\colon \{-1,1\}^n\to\{-1,0,1\},$ using Parseval's identity we obtain
\[
\begin{aligned}
    \|\Delta f\|_{L^1(\{-1,1\}^n)} &\leq \sum_{j=1}^{n} \|\partial_j f\|_{L^1(\{-1,1\}^n)}\leq 2\sum_{j=1}^{n} \|\partial_j f\|_{L^2(\{-1,1\}^n)}^2 = 2\sum_{S \subseteq [n]} |S| |\widehat{f}(S)|^2 \\
    & \leq 2d \sum_{S \subseteq [n]} |\widehat{f}(S)|^2 = 2d \| f\|_{L^2(\{-1,1\}^n)}^2 \leq 2d \| f\|_{L^1(\{-1,1\}^n)}.
\end{aligned}
\]

One may also obtain a counterpart of Theorem \ref{t: main p imp} and Corollary \ref{t: helo} on tail spaces. 
Given $d\in [n]$, let 
\[\Tail_{\ge d}=\{f\colon \{-1,1\}^n\to \mathbb C\colon f=\sum_{|S|\ge d }\widehat{f}(S) \,w_S\}	\]
denote the $d$-th tail space.  In other words,  $f\in\Tail_{\ge d}$ means that $\widehat{f}(S)=0$ if $|S|<d.$
The theorem below is a slightly more general variant of \cite[Theorem 1.3]{HMO} by Heilman--Mossel--Oleszkiewicz.

\begin{prop} 

\label{t: hmo imp}

Fix $p\in (1,\infty)$, let $p_{\varepsilon}$ be given by \eqref{eq: peps} and take $\theta=\theta(p,\varepsilon)$ satisfying \eqref{eq: tpe}. Let  $d\in [n]$ and take $f\in \Tail_{\ge d}$. Then, for any $t>0$,
\begin{equation}
 \label{eq: hmo imp}   
 \|e^{-t\Delta}f\|_{L^p(\{-1,1\}^n)} \le \exp(- c(p,\varepsilon)\cdot td )\|f\|_{L^{2}(\{-1,1\}^n)}^{1-\theta} \|f\|_{L^{p_{\varepsilon}}(\{-1,1\}^n)}^{\theta},
\end{equation}
where $c(p,\varepsilon)$ is a constant that depends only on $p$ and $\varepsilon.$
Consequently, for any $p\in[2,\infty)$ and $\varepsilon>0$ we also have
\begin{equation}
 \label{eq: hmo imp'}   
 \|e^{-t\Delta}f\|_{L^p(\{-1,1\}^n)} \le  \exp(- c(p,\varepsilon)\cdot td )\|f\|_{L^{p+\varepsilon}(\{-1,1\}^n)}.
\end{equation}
\end{prop}
\noindent 
Proposition \ref{t: hmo imp} follows easily from H\"older's inequality without the need to use complex interpolation. Thus its proof is easier than that of Theorem \ref{t: main p imp}. In Section \ref{sec: pmain  imp} we present the argument for the sake of completeness.

We note that taking functions 
$f\colon \{-1,1\}^n\to\{-1,0,1\}$ and $\varepsilon=(p-1)/2$ in Proposition \ref{t: hmo imp}, we obtain a result similar to \cite[Theorem 1.3]{HMO}. Furthermore \eqref{eq: hmo imp}   establishes a weaker version of the heat smoothing conjecture of Mendel--Naor \cite[Remark 5.5]{MN1}.

For related work on Gaussian spaces, we refer to \cite{EI3, EI2, Me1}. A noncommutative analogue of the (analytic) heat smoothing conjecture is given in \cite{Zhang}.

\subsection*{Notation} In the remainder of the paper, for $p\in [1,\infty]$ we abbreviate $L^p:=L^p(\{-1,1\}^n)$ and $\|\cdot\|_{L^p}=\|\cdot\|_p.$

\section{Proofs of  Theorem \ref{t: main p imp} and Proposition \ref{t: hmo imp}}

\label{sec: pmain  imp}

We start with the proof of Theorem \ref{t: main p imp}. Note first that \eqref{eq: main p imp} implies \eqref{eq: main p imp'}. Since in \eqref{eq: main p imp'} we take $p\ge 2$, this is clear from the fact that $p_{\varepsilon}=p+\varepsilon$  and the inequality $\|f\|_2\le \|f\|_{p_{\varepsilon}}.$

Inequality \eqref{eq: main p imp} will follow if we justify that
{\begin{equation}
 \label{eq: main p'}   
 \|(\Delta+\gamma I)^k f\|_{p} \le  C(p,\varepsilon)^k  \,(d+\gamma )^k\, \|f\|_{2}^{1-\theta}\|f\|_{p_{\varepsilon}}^{\theta},
\end{equation}}
uniformly in $\gamma>0$. In the remainder of the proof we focus on \eqref{eq: main p'}   and abbreviate \begin{equation}\label{eq_Ldef}
    L=\Delta+\gamma I.
\end{equation} 
Note that the complex powers $L^{z}$ are well defined for $\Rea z\ge 0$ by the spectral theorem; in our case, they are given explicitly by
\begin{equation}
\label{eq: Lzspect}
L^{z}f=\sum_{S\subseteq [n]} (|S|+\gamma)^z\, \widehat{f}(S)\,w_S.
\end{equation}

The operator $L$ in \eqref{eq_Ldef} is clearly self-adjoint on $L^2.$ Moreover, it is not hard to see that $e^{-tL}$ is a contraction semigroup on all $L^p$ spaces, $1<p<\infty.$ Indeed, since $e^{-tL}=e^{-\gamma t} e^{-t\Delta}$ we have the pointwise domination $|e^{-tL}f(x)|\le e^{-t\Delta} |f|(x).$ Furthermore, recalling the definitions \eqref{eq:hala}--\eqref{eq:pala} and denoting $$B f(\delta)= \frac12 \sum_{j=1}^n f(\delta_1,\ldots,-\delta_j,\ldots,\delta_n),$$ we may rewrite
$\Delta f= \frac{n}{2} f -Bf.$
Since $\|Bf\|_p\le \frac{n}{2}\|f\|_p$ for $p\in [1,\infty]$, we obtain
\[
\|e^{-tL} f\|_p\le  \|e^{-t\Delta}f\|_p = e^{-nt/2} \|e^{tB} f\|_p\le e^{-nt/2} e^{nt/2}\|f\|_p=\|f\|_p,
\]
where the last inequality above comes from expanding $e^{tB}$ into a power series and estimating it term by term. In summary, we verified that $L$ generates a symmetric contraction semigroup in the sense of Cowling's \cite{Co}. Hence, we may use  \cite[Corollary 1]{Co}  in order to obtain an estimate for the imaginary powers $L^{iu},$ $u\in \R.$
\begin{prop}
    \label{pro: imp}
    Let $u\in\R.$
    For each $1<p<\infty$ there exists a constant $C_p>1,$ depending only on $p$, such that
    \begin{equation}
\label{eq: imp}
\|L^{iu}f\|_{p} \le C_p \cdot \exp((\pi |1/p-1/2|+1)|u|)\|f\|_p,
    \end{equation}
    for all $f\in L^p$.
\end{prop}
\noindent The above proposition is slightly weaker than \cite[Corollary 1]{Co} but  will suffice for our purposes. {The $L^p$ boundedness of $L^{iu}$ is also a consequence of an earlier result of Stein; see Corollary 3 in \cite[Chapter IV, Section 6, p.\ 121]{topics}, however, no explicit estimate for the norm is given there.}

We are now ready to prove {\eqref{eq: main p'}. 
\begin{proof}[Proof of \eqref{eq: main p'}]
Given $p\in (1,\infty)$, we continue to denote by $p'$ its dual exponent. By duality, it is enough to justify that, for all $\varepsilon>0$,
\begin{equation}
\label{eq:dual}
|\langle L^k f,g\rangle|\leq  C(p,\varepsilon)^k  \,(d+\gamma )^k\, \|f\|_{2}^{1-\theta}\|f\|_{p_{\varepsilon}}^{\theta}
\end{equation}
uniformly for all real-valued functions $f$ and $g$ such that $\|g\|_{p'}=1.$ In what follows, we fix such functions $f$ and $g$. Denote $$q:=p_{\varepsilon}$$ and recall that $\theta=\theta(p,\varepsilon)$ is such that
\begin{equation}
\label{eq: ptq}
   \frac{1}{p}=\frac{\theta}{q}+\frac{1-\theta}{2}.
 \end{equation}

Decompose $g=\sum_{\delta\in \{-1,1\}^n} g(\delta)\ind{\delta}$ and, recalling that $g$ is real-valued, let
\begin{equation}\label{eq_g_zSum}
g_z=\sum_{\delta\in \{-1,1\}^n} \sgn(g(\delta))|g(\delta)|^{p'(1-z)/2+p'z/q'} \ind{\delta},\qquad z\in \CC,    
\end{equation}
where $q'$ is the dual exponent of $q.$ The exact formula for $g_z$ is immaterial and we will only need to use some of its properties. Firstly, note that for fixed $s$ the function $g_z(s)$ is holomorphic in $z$. 
Moreover, by  \eqref{eq: ptq} we have
$1/p'=(1-\theta)/2+\theta /q',$
so 
\begin{equation}
\label{eq:gzpro}
g_{\theta}=g,\qquad |g_z|=|g|^{p'/2} \textrm{ if } \Rea z=0,\qquad |g_z|=|g|^{p'/q'} \textrm{ if } \Rea z=1.
\end{equation}
Let $\Sigma=\{z\in \CC\colon 0\le \Rea z \le 1\}$ and set
\begin{equation}
\label{eq: N}
N=\frac{k}{1-\theta}.
\end{equation}
Given $z\in \Sigma$, we define the function
\begin{equation}
\label{eq: vpz}
\varphi(z)=\exp(z^2 N\pi |1/q-1/2|) \langle L^{N(1-z)} f,\overline{g_z}\rangle.
\end{equation}
Since the sum in \eqref{eq_g_zSum}  is finite, the function $\varphi$ is holomorphic in  the interior of $\Sigma.$ It also follows that $|\varphi(z)|$ is bounded in $\Sigma,$ i.e.\ $|\varphi(z)|\le C(N,\gamma,q,p,f,g).$ Thus we may apply Hadamard's three-lines theorem to the function $\varphi$,  provided we appropriately estimate its boundary values at $\Rea z=0$ and $\Rea z=1.$ We claim that
\begin{equation}
\label{eq: goal1}
|\varphi(z)|\le   (d+\gamma)^N \|f\|_2,\qquad \Rea z=0, \text{ and }
\end{equation}
\begin{equation}
\label{eq: goal2}
|\varphi(z)|\le  (C_{p,q})^{k}\|f\|_q,\qquad \Rea z=1,
\end{equation}
for a constant $C_{p,q}>1$ depending only on $p$ and $q.$ 

Assuming the claim for the moment, we apply Hadamard's three-lines theorem \cite[Lemma 1.3.5]{gra} with \[B_0= (d+\gamma)^N \|f\|_2,\qquad B_1=  (C_{p,q})^{k}\|f\|_q,\]
and obtain (recall \eqref{eq: N})
\[
|\varphi(\theta)|\le  (C_{p,q})^{k\theta} (d+\gamma)^k \|f\|_2^{1-\theta} \|f\|_q^{\theta}.
\]
Since $q=p_{\varepsilon}$ the constant $C_{p,q}$  depends only on $p$ and $\varepsilon.$ At this point, to complete the proof of \eqref{eq:dual} we note that \eqref{eq: N}  and \eqref{eq:gzpro} imply 
\[
\langle L^k f,g\rangle=\exp(-\theta^2 N\pi |1/q-1/2|)\varphi(\theta).\]

We are left with proving the claimed inequalities \eqref{eq: goal1} and \eqref{eq: goal2}. We start with a proof of \eqref{eq: goal1}. Take $z=iu,$ $u\in \R$, and note that $\|L^{N(1-iu)} f\|_2\le (d+\gamma)^{N}\|f\|_2.$ Since $f\in \Deg_{\le d},$ this inequality follows from \eqref{eq: Lzspect} and the orthogonality of the Walsh functions $w_S$. We remark that this is the only place where the assumption  $f\in \Deg_{\le d}$ is used.
Now, by the Cauchy--Schwarz inequality and the second identity in \eqref{eq:gzpro}, we obtain
\[
|\varphi(z)|\le \exp(-u^2 N\pi |1/q-1/2|)(d+\gamma)^N\|f\|_2, 
\]
which is even better than the claimed \eqref{eq: goal1}.

Finally, we prove \eqref{eq: goal2}. Take $z=1+iu,$ $u\in \R$, and apply H\"older's inequality with exponents $q$ and $q'$ to the formula \eqref{eq: vpz} defining $\varphi.$ This yields
\[
|\varphi(z)|\le \exp((1-u^2) N\pi |1/q-1/2|)\|L^{-Niu} f\|_q, 
\]
where we used the third identity in  \eqref{eq:gzpro}. Applying Proposition \ref{pro: imp} with $q$ in place of $p$ and using \eqref{eq: N}, we obtain
\begin{align*}
|\varphi(z)|&\le C_q \exp\left[(1-u^2) N\pi |1/q-1/2|+N(\pi|1/q-1/2|+1)|u|\right]\| f\|_q\\
&= C_q (B_{p,q}(u))^k\| f\|_q,
\end{align*}
where $C_q>1$ is a constant that depends only on $q$, and  $$B_{p,q}(u):=\exp\left[\frac{(1-u^2)\pi |1/q-1/2|+(\pi|1/q-1/2|+1)|u|}{1-\theta}\right].$$ Denoting by $B^{\max}_{p,q}$ the maximal value of $B_{p,q}(u)$, we reach
\begin{equation}
|\varphi(z)|\le (C_q B^{\max}_{p,q})^k \| f\|_q.
\end{equation}
This proves that \eqref{eq: goal2} holds with $C_{p,q}=C_q B^{\max}_{p,q}$ and completes the proof of Theorem \ref{t: main p imp}.
\end{proof}
}

{
We now move to the proof of Proposition \ref{t: hmo imp}.
\begin{proof}[Proof of Proposition \ref{t: hmo imp}]
    The argument that \eqref{eq: hmo imp} implies \eqref{eq: hmo imp'} is analogous to the proof that \eqref{eq: main p imp} implies \eqref{eq: main p imp'} given at the start of this section.

    To justify \eqref{eq: hmo imp}, we apply H\"older's inequality and obtain
    \[
\|e^{-t\Delta}f\|_{p}\le \|e^{-t\Delta}f\|_{2}^{1-\theta} \|e^{-t\Delta}f\|_{p_{\varepsilon}}^{\theta}.
    \]
Since $f\in \Tail_{\ge d}$ Parseval's identity yields $\|e^{-t\Delta}f\|_{2}\le \exp(-td)\|f\|_2.$ Thus, using the contractivity of $\exp(-t\Delta)$ on $L^{p_{\varepsilon}}$, we obtain
\[
\|e^{-t\Delta}f\|_{p}\le \exp(-(1-\theta)\cdot td)\|f\|_{2}^{1-\theta} \|f\|_{p_{\varepsilon}}^{\theta}.
\]
This gives \eqref{eq: hmo imp} with $c(p,\varepsilon)=1-\theta(p,\varepsilon)$ and $\theta(p,\varepsilon)$ as defined in \eqref{eq: tpe}.
\end{proof}
}

\section{Counterexamples: proofs of Propositions \ref{lem: count 1} and \ref{lem: count 2}}

\label{sec: pcon}
{We start with the proof of Proposition \ref{lem: count 1}, which is essentially a repetition of the proof of Theorem 5 from \cite{EI1}.

\begin{proof}[Proof of Proposition \ref{lem: count 1}]
Denote by $T_d$ the $d$-th Chebyshev polynomial of the first kind and let  $f$ be the function
\begin{equation}\label{eq_defF+2}
f(\delta)=T_d\left(\frac{\delta_1+\cdots +\delta _n}{n}\right),\qquad \delta \in \{-1,1\}^n.    
\end{equation}
Since $T_d$ is a polynomial of degree $d$ that takes values in  $[-1,1]$ we have $f\in \Deg_{\le d}$ and $\|f\|_{\infty}\le 1$.

Using identities \cite[eqs.\ (89)--(91)]{EI1} with $\varepsilon=(1,\ldots,1)$, we see that
\[
(\Delta f) (1,\ldots,1)=\frac{n}{2}\left(T_d(1)-T_d(1-\frac{2}{n})\right),
\]
and consequently
\[\|\Delta f\|_{\infty}\ge \frac{n}{2}\left(T_d(1)-T_d(1-\frac{2}{n})\right) \ge  \frac{n}{2}\left(T_d(1)-T_d(1-\frac{2}{n})\right)\cdot \|f\|_{\infty}.\]
Since $T_d'(1)=d^2,$ letting $n\to \infty$ we  have $\lim_{n\to \infty} {\bf C}(n,d)\ge d^2$, and the proof is complete. 
\end{proof}
}

\begin{proof}[Proof of Proposition \ref{lem: count 2}]

The proof is a translation of the Kushilevitz function from \cite[Section 6.3]{HKP} into our setting. We need to change the Hamming cube labelled by bits $0$ and $1$ to the one we use and to realize that the translated function has the desired properties. 

Given two boolean functions $\tf\colon \{0,1\}^m\to \{0,1\}$ and $\tg\colon \{0,1\}^n \to \{0,1\}$, we define their composition (which is a function of $mn$ variables) via
\[
(\tf\diamond \tg) (x_{11},\ldots, x_{mn})=\tf(\tg(x_{11},\ldots,x_{1n}),\ldots,\tg(x_{m1},\ldots,x_{mn})).
\]
Let $n=6^k$ and $d=3^k.$ The Kushilevitz function on the set $\{0,1\}^n$ is defined via the $k$-fold composition $\tf_k=\th\diamond \th\diamond \cdots \diamond \th.$ The function $\th$ being composed is
\[
\th(x_1,\ldots,x_6)=\sum_{i=1}^6 x_i -\sum_{{i,j}\in {[6] \choose 2}} x_i x_j + \sum_{\{i,j,k\} \in K} x_i x_j x_k,
\]
where $[6] \choose 2$ denotes the set of all $2$ element subsets of $\{1,\ldots,6\}$ while 
\begin{align*}
K=\{&
\{1,2,5\},\{1,2,6\}, \{1,3,4\}, \{1,3,6\},
\{1,4,5\},\\& \{2,3,4\}, \{2,3,5\}, \{2,4,6\}, \{3,5,6\}, \{4,5,6\}
\}.
\end{align*}
A straightforward but somewhat cumbersome  case distinction reveals that the function $\th$ is \{0,1\}-valued.
We refer to \cite[Section 6.3]{HKP} or \cite[Section 3.1]{AW} for further details.

Let $\tf\colon \{0,1\}^n \to \{0,1\}$ be a function. We say that a coordinate (bit) $i$ is {\it sensitive} for $x\in \{0,1\}^n$ if flipping the $i$-th bit results in flipping the output of $\tf$. The sensitivity of $\tf$ on the input $x,$ denoted by $s(\tf,x),$ is the number of bits that are sensitive for $\tf$ on the input $x.$ The sensitivity of $\tf$ is defined as the maximum $s(\tf)=\max_{x\in \{0,1\}^n} s(\tf,x).$ It turns out that $s(\tf_k)=6^k=n$; see \cite[Claim 3.2.1]{AW}. We also have $\deg(\tf_k)=3^k$; see \cite[Claim 3.2.2]{AW}. Here,  $\deg(\tf_k)$ denotes the degree of $\tf_k$ as a multilinear polynomial.

Now, for each $k\in \NN$ define the function $f_k\colon \{-1,1\}^n\to \{-1,1\}$ via
\[
f_k(\delta)=2\tilde f_k \left(\frac{\delta_1+1}{2},\ldots,\frac{\delta_n+1}{2}\right)-1,\qquad \delta \in \{-1,1\}^n.
\]
Defining sensitivity for $f_k$ exactly as for $\tf_k$, it is clear that $s(f_k)=s(\tf_k)=6^k.$  The second formula in \cite[Proposition 2.37]{RO} reveals that $\|\Delta f_k\|_{\infty}=s(f_k)=6^k.$ Since $\deg (f_k)=3^k$, we have $f_k \in \Deg_{\le 3^k}$, and thus
\[
\|\Delta f_k\|_{\infty}=6^k=3^{k\log 6 / \log 3}=d^{\log 6/\log 3}\|f_k\|_{\infty}.
\]
This completes the proof of Proposition \ref{lem: count 2}.
\end{proof}

\subsection*{Acknowledgments} This project grew out of the workshop {\it Analysis on the hypercube with applications to quantum computing} held at the American Institute of Mathematics (AIM) in June 2022. We thank the organizers of the workshop and the staff of AIM for creating a superb environment for collaboration; in particular, we are grateful to Irina Holmes Fay who was a member of our research group. We also thank Paata Ivanisvili and Alexandros Eskenazis for helpful remarks and suggestions. We are grateful to the anonymous referee for helpful remarks that led to significant improvements to the presentation and for pointing out an oversight in one of our previous arguments. 

PD is   supported by the grant NSF  DMS-2154356. OK is supported in part by a grant from the Israel Science Foundation (ISF Grant No. 1774/20), and by a grant from the US-Israel Binational Science Foundation and the US National Science Foundation (BSF-NSF Grant No. 2020643). DOS  is partially supported by FCT/Portugal through CAMGSD, IST-ID, projects UIDB/04459/2020 and UIDP/04459/2020 and
 by IST Santander Start Up Funds.
LS is supported by the Primus research programme PRIMUS/21/SCI/002 of Charles University and by Charles University Research Centre program No.\ UNCE/24/SCI/005.
BW is supported by the National
Science Centre, Poland, grant Sonata Bis 2022/46/E/ST1/ 00036. For the purpose of Open Access, the authors have applied a CC-BY public copyright licence to any
Author Accepted Manuscript (AAM) version arising from this submission.


\end{document}